\documentclass[article,10pt,a4paper]{article}

\usepackage{amssymb,amsmath,amsfonts}

\usepackage[utf8]{inputenc}
\usepackage[cyr]{aeguill}
\usepackage[francais]{babel}

\usepackage{natbib}

\usepackage{enumerate}

\usepackage{latexsym}
\usepackage[scaled=0.9]{helvet}

\usepackage{graphicx}
\usepackage{epstopdf}
\usepackage{comment}

\usepackage{url}
\usepackage[hidelinks]{hyperref}
\hypersetup{
    bookmarks=true,         
    unicode=false,          
    pdftoolbar=true,        
    pdfmenubar=true,        
    pdffitwindow=false,     
    pdfstartview={FitH},    
    pdftitle={My title},    
    pdfauthor={Author},     
    pdfsubject={Subject},   
    pdfkeywords={keyword1} {key2} {key3}, 
    pdfnewwindow=true,      
    colorlinks=black,       
    linkcolor=black,          
    citecolor=black,        
    filecolor=black,      
    urlcolor=black           
}

\newtheorem{nntheorem}{\bf\fontfamily{phv}\selectfont Théorème}
\newtheorem{nnassumption}{\bf\fontfamily{phv}\selectfont Hypothèse}
\newtheorem{nndefinition}{\bf\fontfamily{phv}\selectfont Definition}
\newtheorem{nnlemma}[nndefinition]{\bf\fontfamily{phv}\selectfont Lemma}
\newtheorem{nncorollary}[nndefinition]{\bf\fontfamily{phv}\selectfont Corollary}
\newtheorem{nnproposition}[nndefinition]{\bf\fontfamily{phv}\selectfont Proposition}
\newtheorem{nexample}[nndefinition]{\bf\fontfamily{phv}\selectfont Exemple}

\newenvironment{theorem}
{\begin{nntheorem}\it}
{\end{nntheorem}}
\newenvironment{proposition}
{\begin{nnproposition}\it}
{\end{nnproposition}}

\newenvironment{definition}
{\begin{nndefinition}\it}
{\end{nndefinition}}
\newenvironment{assumption}
{\begin{nnassumption}\it}
{\end{nnassumption}}

\newenvironment{nnexample}
{\begin{nexample}\rm}{\end{nexample}}

\newtheorem{nnremark}{\bf\fontfamily{phv}\selectfont Remarque}
\newenvironment{remark}{\begin{nnremark}}{\hfill \hspace*{1pt}\hfill $\circ$\end{nnremark}}

\newenvironment{notation}{{\bf \fontfamily{phv}\selectfont Notations.}}{\hfill \hspace*{1pt}\hfill}

\newenvironment{proof}{{\it \fontfamily{phv}\selectfont Démonstration.\ }}{\hfill \hspace*{1pt}\hfill $\bullet$}

\usepackage{caption}
\DeclareCaptionFormat{myformat}{\fontfamily{phv}\selectfont#1#2#3\par}
\newcommand{\dx}{\dot{x}}
\newcommand{\class}{\mathcal{C}}
\newcommand{\R}{\mathbb{R}}
\newcommand{\Ras}{\mathbb{R}_{\geq 0}}

\newcommand{\Kinf}{\mathcal{K}^{\infty}}
\newcommand{\A}{\mathbf{A}}

\usepackage{authblk}

\title{Union d'une commande par \emph{backstepping} avec une commande locale}
\date{}
\author[1]{Humberto Stein Shiromoto \thanks{E-mail: \texttt{humberto.shiromoto@ieee.org}.}}
\author[2]{Vincent Andrieu}
\author[1]{Christophe Prieur}

\affil[1]{\small Gipsa-lab, 11 rue des Mathématiques, Grenoble Campus, BP 46, F - 38402 SAINT MARTIN D'HERES, Cedex,  France}
\affil[2]{\small LAGEP-CNRS, Université Claude Bernard Lyon 1, bât 308G ESCPE-Lyon, 2ème étage, 43 bd du 11 Novembre 1918, 69622 Villeurbanne Cedex FRANCE}

\begin{document}

\maketitle

\begin{abstract}
.\ On considère une classe de systèmes non-linéaires pour lesquels la technique de synthèse par \emph{backstepping} n'est pas applicable. On présente un critère pour la synthèse d'une commande hybride faisant l'union d'une commande par backstepping et d'une commande locale. Cette commande hybride rend le système bouclé globalement asymptotiquement stable. Le critère de sélection repose sur le choix de la taille et la stabilisation d'un ensemble inclus dans le bassin d'attraction d'une commande locale. Les résultats sont illustrés par des simulations.
\end{abstract}

\section{Introduction}

Il existe de nombreuses méthodes différentes pour la synthèse de commandes pour les systèmes non-linéaires (voir par exemple, \citep{FreemanKokotovic2008,Khalil1992,Kokotovic1992}).
En fonction de la structure de chaque système, la synthèse suit une approche différente, par exemple, avec un grand gain  \citep{GrognardSepulchreBastin1999}, par \emph{forwarding} \citep{JankovicSepulchre96,MazencPraly96,SepulchreJankovic97} et par \emph{backstepping} \citep{FreemanKokotovic2008,KrsticKokotovicKanellakopoulos95,PralyAndreaCoron91}. En raison de la présence des paramètres inconnus  ou de la structure dynamique, certaines techniques peuvent être inapplicables.

Pour le type de systèmes où la technique de synthèse par \emph{backstepping} peut ne pas être utilisable, l'approche présentée dans  \citep{5991251} donne une direction pour obtenir une commande stabilisante globale. Plus précisément, la technique se propose de rassembler deux types de commandes. L'union de commandes ainsi que de fonctions de Lyapunov a déjà étudiée dans la littérature, voir \citep{Andrieuetal2011}, \citep{PrieurAndrieu2010} et \citep{Prieur2001}. Dans cette dernière référence cette union a été réalisée  en utilisant des commandes hybrides (\citep{HespanhaLiberzonMorse99}, \citep{MorinSamson00}).
Un des aspects essentiels est  que cette classe de commandes hybrides est robuste par rapport aux erreurs de mesure et aux erreurs d'actionneur.

Il est important de noter que dans \citep{5991251}, une classe de systèmes non-linéaires a été présentée pour illustrer l'utilisation de cette technique. En revanche, cette approche laisse en suspens des questions ouvertes. Ainsi, il n'a été défini aucune méthode pour la détermination des gains des commandes obtenues. Dans cet article en reprenant l'exemple particulier étudié dans \citep{5991251}, nous montrons comment un choix approprié des paramètres nous permet de concevoir une loi de commande hybride qui stabilise globalement et asymptotiquement l'origine.
Par ailleurs, on présente aussi des limitations sur cette technique, plus précisément on analyse un cas tel que  l'hypothèse d'inclusion (voir hypothèse \ref{hyp:inclusion}, introduite ultérieurement) n'est pas satisfaite.

Dans la section \ref{sec:formulation du probleme} les concepts sur lesquels nous travaillons et la formulation du problème en considération seront introduits. Dans la section \ref{sec:classe de problemes} nous présentons le résultat de ce travail pour la classe de systèmes non-linéaires considérée. Nous illustrons le résultat avec des simulations présentées dans la section \ref{sec:simulations}.

\begin{notation} Dans cet article nous utilisons les notations suivantes: $x\cdot y$ est le produit interne cartésien entre deux vecteurs $x$ et $y$, $||\cdot||$ dénote la norme induite. La boule fermée et unitaire est notée par $\overline{\mathbf{B}}$. La classe des fonctions $f:\Ras\to\Ras$ qui sont continues, zéro à zéro, strictement croissantes et non bornées est notée par $\Kinf$. En ce qui concerne les dérivées, on note par $\partial_x f(x)$  la dérivée d'une fonction $f$ par rapport à un vecteur $x$. La dérivée de Lie d'une fonction $V$ par rapport au vecteur $f$, c'est-à-dire $\partial_xV(x)\cdot f(x,u)$ calculée en $(x,u)$ sera notée par $L_fV(x,u)$. Soit $\dx=f(x,u(x))$ un système en boucle fermée, son bassin d'attraction sera noté par $\mathbf{B}_f(u)$.
\end{notation}

\section{Formulation du problème}\label{sec:formulation du probleme}
Considérons le système non-linéaire défini par
\begin{equation}\label{eq:systeme generale}
\left\{\begin{array}{rcl}
\dx_1&=&f_1(x_1,x_2)+h_1(x_1,x_2,u)\\
\dx_2&=&f_2(x_1,x_2)u+h_2(x_1,x_2,u)
\end{array}\right.
\end{equation}
avec $(x_1,x_2)\in\R^{n-1}\times\R$ et $u\in\R$ la commande. Les fonctions $f_1$, $f_2$, $h_1$ et $h_2$ sont localement Lipschitziennes. Par ailleurs, les fonctions décrivant la dynamique satisfont à l'origine $f_1(0,0)=h_1(0,0,0)=0$, $h_2(0,0,0)=0$ et $f_2(x_1,x_2)\neq0$, $\forall(x_1,x_2)\in\R^n$.

On utilise aussi une notation plus compacte pour le système \eqref{eq:systeme generale}: $\dx=f_h(x,u)$. Dans le cas $h_1\equiv0$ et $h_2\equiv0$, on note simplement $\dx=f(x,u)$.

\subsection{Des hypothèses}
\begin{assumption}[Stabilité locale]\label{hyp:stabilite locale}
Il existe une fonction $\class^1$ positive, définie et propre $V_\ell:\R^n\to\Ras$, une fonction continue $\varphi_\ell:\R^n\to\R$ et une constante positive $c_\ell$ telles que, $\forall x\in\Omega_{c_\ell}(V_\ell)$,
\begin{equation}\label{eq:hypothese:1}
L_{f_h}V_\ell(x,\varphi_\ell)<0,
\end{equation}
avec $\Omega_{c_\ell}(V_\ell)=\{x:V_\ell(x)<c_\ell,\,x\neq0\}$.
\end{assumption}

Pour vérifier cette hypothèse, il est possible de faire la synthèse d'une commande locale pour \eqref{eq:systeme generale} en exploitant une approximation du modèle. Par exemple, si l'approximation de premier ordre autour de l'origine est contrôlable, alors il est possible d'obtenir une commande linéaire locale et une fonction de Lyapunov quadratique telles que l'hypothèse \ref{hyp:stabilite locale} soit satisfaite.

À cause de la présence des fonctions $h_1$ et $h_2$ et leurs dépendances par rapport à la commande $u$, il n'est pas possible d'utiliser la technique de \emph{backstepping} dans le sens classique. Plus précisément, supposons que l'on ait une commande $\varphi_1$ stabilisante pour le sous-système $\dx_1=f(x_1,x_2)$ et une fonction de Lyapunov $V_1$ associée (autrement dit que l'item a) de l'hypothèse \ref{hyp:bornes} dessous est valide), alors, avec la procédure classique du \emph{backstepping} la fonction $V(x_1,x_2)=V_1(x_1)+(x_2-\varphi_1(x_1))^2/2$ est une fonction de Lyapunov candidate pour $\dx=f_h(x,\varphi_1(x))$. On calcule $\dot{V}(x):=L_{f_h}V(x_1,x_2,u)$, $\forall (x_1,x_2,u)\in\R^{n-1}\times\R^n\times\R$,
\begin{equation*}
\begin{split}
\dot{V}(x)\leq-\alpha(V_1(x_1))+(x_2-\varphi_1(x_1))\cdot[f_2(x_1,x_2)\\
\cdot u+h_2(x_1,x_2,u)-L_{f_1+h_1}\varphi_1(x_1,x_2,u)\\
+\partial_{x_1}V_1(x_1)\cdot\int_0^1f_1(x_1,sx_1-(1-s)\varphi_1(x_1))\,ds]\\
+L_{h_1}V_1(x_1,x_2,u).
\end{split}
\end{equation*}
Soit
\begin{equation*}
\begin{split}
E(x_1,x_2,u)=[f_2(x_1,x_2)u+h_2(x_1,x_2,u)\\
-L_{f_1+h_1}\varphi_1(x_1,x_2,u)\\
+\partial_{x_1}V_1(x_1)\cdot\int_0^1f_1(x_1,sx_1-(1-s)\varphi_1(x_1))\,ds].
\end{split}
\end{equation*}
Alors,
\begin{equation}\label{eq:backstepping classique}
\begin{split}
\dot{V}(x)\leq-\alpha(V_1(x_1))+(x_2-\varphi_1(x_1))\\
\cdot E(x_1,x_2,u)+L_{h_1}V_1(x_1,x_2&,u)
\end{split}
\end{equation}
Par conséquence, pour obtenir un terme $(x_2-\varphi_1(x_1))^2$ à droite de l'inégalité \eqref{eq:backstepping classique}, il est nécessaire de résoudre le système d'équations implicites en $u$ donné par
$$\begin{array}{rcl}
L_{h_1}V_1(x_1,x_2,u)&=&-k(x_2-\varphi_1(x_1))^2\\
E(x_1,x_2,u)&=&-k(x_2-\varphi_1(x_1))
\end{array}$$
avec $k>0$. Ce système d'équations implicites est en général dificile à résoudre. Cette obstruction nous incite à introduire des hypothèses sur les fonctions $h_1$ et $h_2$ dans l'objectif de résoudre le problème de stabilisation. 

\begin{assumption}[Bornes]\label{hyp:bornes}
Il existe une fonction $V_1:\R^{n-1}\to\Ras$ positive, définie, propre et de classe $\class^1$. Une fonction $\varphi_1:\R^{n-1}\to\R$ de classe $\class^1$. Une fonction $\alpha:\Ras\to\Ras$ localement Lipschitz de classe $\Kinf$. Une fonction continue $\Psi:\R^n\to\R$ et deux constantes positives $\varepsilon<1$ et $M$ telles que les propriétés ci-dessous sont valides
\begin{enumerate}[a)]
\item (Contrôleur stabilisant pour le sous-système en $x_1$) $\forall x_1\in\R^{n-1}$,
$$L_{f_1}V_1(x_1,\varphi_1(x_1))\leq-\alpha(V_1(x_1))$$

\item (Borne sur $h_1$) $\forall (x_1,x_2,u)\in\R^{n-1}\times\R\times\R$,
$$\begin{array}{rcl}
||h_1(x_1,x_2,u)||&\leq&\Psi(x_1,x_2),\\
L_{h_1}V_1(x_1,\varphi_1(x_1),u)&\leq&(1-\varepsilon)\alpha(V_1(x_1))+\varepsilon\alpha(M);
\end{array}$$

\item (Borne sur $\partial_{x_2}h_1$) $\forall (x_1,x_2,u)\in\R^{n-1}\times\R\times\R$,
$$\left|\left|\partial_{x_2} h_1(x_1,x_2,u)\right|\right|\leq\Psi(x_1,x_2);$$

\item (Borne sur $h_2$) $\forall (x_1,x_2,u)\in\R^{n-1}\times\R\times\R$,
$$||h_2(x_1,x_2,u)||\leq\Psi(x_1,x_2)$$
\end{enumerate}
\end{assumption}

La condition a) de l'hypothèse \ref{hyp:bornes} permet de concevoir une commande stabilisante globale pour l'équation définie par
\begin{equation}\label{eq:sous systeme general}
\dx_1=f_1(x_1,x_2),
\end{equation}
lorsque $x_2$ est vu comme une entrée. C'est-à-dire que cette condition permet de concevoir une commande pour $\dx=f(x,u)$ à travers la technique de \emph{backstepping}. Les conditions b)-d) fournissent des bornes pour les termes $h_1$ et $h_2$ ne permettant pas que cette technique soit appliquée pour le système \eqref{eq:systeme generale}. En exploitant les bornes sur les fonctions $h_1$ et $h_2$, nous arrivons à concevoir une classe de commandes qui garantit que l'ensemble compact donné par
\begin{equation}\label{eq:general set:A}
\begin{split}
\A=\{(x_1,x_2)\in\R^n:V_1(x_1)\leq M,x_2=\varphi_1&(x_1)\}.
\end{split}
\end{equation}
est globalement pratiquement asymptotiquement stable. Plus précisément, nous avons le résultat suivant présenté dans \citep[Proposition 3.1]{5991251}
\begin{proposition}\label{prop:commande globale}
Si l'hypothèse \ref{hyp:bornes} est satisfaite, alors l'ensemble $\A$ défini par \eqref{eq:general set:A} est  globalement pratiquement stable, c'est-à-dire,  pour chaque $a>0$, il existe une commande continue $\varphi_g$ telle que $\A+a\overline{\mathbf{B}}$ contient un ensemble qui est globalement asymptotiquement stable pour le système $\dx=f_h(x,\varphi_g(x))$.
\end{proposition}

La démonstration de la Proposition \ref{prop:commande globale} fournit une commande $\varphi_g$ dont la définition est:
\begin{equation}\label{eq:commande globale}
\begin{split}
\varphi_g(x_1,x_2)=\frac{1}{K_V f_2(x_1,x_2)}\cdot[K_V\\
L_{f_1}\varphi_1(x_1,x_2)-\partial_{x_1}V_1(x_1)\\
\cdot\int_0^1\partial_{x_2}f_1(x_1,\eta_{x_1,x_2}(s))\,ds\\
-(x_2-\varphi_1(x_1))\cdot(c+\frac{c}{4}\Delta^2(x_1,x_2))],
\end{split}
\end{equation}
avec $K_V=2(M+a)/a^2$, $\eta_{x_1,x_2}(s)=sx_2+(1-s)\varphi_1(x_1)$, $c$ une constante positive et
\begin{equation*}
\begin{split}
\Delta(x_1,x_2)=||\partial_{x_1}V_1(x_1)||\int_0^1\Psi(x_1,\eta_{x_1,x_2}(s))\,ds\\
+\Psi(x_1,x_2)K_V(1+||\partial_{x_1}\varphi_1(x_1)||)
\end{split}
\end{equation*}

Afin d'obtenir une loi de commande qui stabilise globalement et asymptotiquement l'origine, il faut donc rassembler la commande locale $\varphi_\ell$ donnée par l'hypothèse \ref{hyp:stabilite locale} et la commande obtenue à la Proposition \ref{prop:commande globale} en exploitant un formalisme hybride. Pour cela nous utilisons l'hypothèse suivante.

\begin{assumption}[Inclusion]\label{hyp:inclusion}
La fonction $V_\ell$ satisfait l'inégalité suivante
$$\displaystyle\max_{(x_1,x_2)\in\A}V_\ell(x_1,x_2)<c_\ell.$$
\end{assumption}
Les hypothèses \ref{hyp:stabilite locale} et \ref{hyp:inclusion} imposent que l'ensemble $\A$ soit  inclus dans le bassin d'attraction de \eqref{eq:systeme generale} en boucle fermée avec $\varphi_\ell$.

Dans \citep{5991251}, nous avons présenté le problème de conception d'une commande pour \eqref{eq:systeme generale} de telle sorte que l'origine soit globalement asymptotiquement stable pour le système \eqref{eq:systeme generale} en boucle fermée. Nous présentons ici la définition de ce type de commande (pour une lecture plus détaillée sur les systèmes hybrides, le lecteur est invité à consulter \citep{4806347} et \citep{4380508})
 \begin{definition}
 Une commande hybride pour \eqref{eq:systeme generale}, notée par $\mathbb{K}$, est composée de
 \begin{itemize}
 \item Un ensemble totalement ordonné $Q$;

 \item Pour chaque $q\in Q$
 \begin{itemize}
 \item des ensembles fermés $C_q\subset\R^n$ et $D_q\subset\R^n$ tels que $C_q\cup D_q=\R^n$;

 \item des fonctions continues $\varphi_q:C_q\to\R$;

 \item des fonctions extérieurement semi-continues, localement bornées, uniformément dans $Q$,  évaluées en ensembles $G_q:D_q\rightrightarrows Q$ avec des images non vides.
 \end{itemize}
 \end{itemize}
 \end{definition}
Avec ces données, nous pouvons définir une commande hybride telle que l'origine du système bouclé est globalement asymptotiquement stable \citep[theorème 1]{5991251}:

\begin{theorem}\label{thm:commande hybride}
Soient $c_\ell$ et $\tilde{c}_\ell$ deux constantes positives qui satisfont $0<\tilde{c}_\ell<c_\ell$.
Si les hypothèses \ref{hyp:stabilite locale}-\ref{hyp:inclusion} sont toutes valides et s'il existe une constante $a>0$ telle que la commande hybride $\mathbb{K}$ définie par $Q=\{1,2\}$, les sous-ensembles
$$\begin{array}{rcl}
C_1&=&\{(x_1,x_2)\in\R^{n-1}\times\R:V_\ell(x_1,x_2)\leq c_\ell\}\\
C_2&=&\{(x_1,x_2)\in\R^{n-1}\times\R:V_\ell(x_1,x_2)\geq \tilde{c}_\ell\}\\
D_q&=&\overline{(\R^{n-1}\times\R)\setminus C_q},\forall q\in Q,
\end{array}$$
les commandes
$$\begin{array}{rrcl}
\varphi_1:&C_1&\to&\R^{n-1}\times\R\\
&(x_1,x_2)&\mapsto&\varphi_1(x_1,x_2)=\varphi_\ell(x_1,x_2),\\
\\
\varphi_2:&C_2&\to&\R^{n-1}\times\R\\
&(x_1,x_2)&\mapsto&\varphi_2(x_1,x_2)=\varphi_g(x_1,x_2)
\end{array}
$$
et les applications multi-valuées et définies par $D_q\ni (x_1,x_2)\mapsto G_q(x_1,x_2)=\{3-q\}$, $q\in Q$ rend l'origine globalement asymptotiquement stable pour le système \eqref{eq:systeme generale} en boucle fermée avec $\mathbb{K}$:
$$\left\{\begin{array}{rclrcl}
\dx&=&f_h(x,\varphi_q(x)),&x&\in&C_q\\
q^+&\in&G_q(x),&x&\in&D_q
\end{array}\right.$$
\end{theorem}

\begin{remark}
Si les items b)-d) de l'hypothèse \ref{hyp:bornes} sont valides, l'équation \eqref{eq:backstepping classique}
témoigne de l'effet des fonctions $h_1$ et $h_2$ sur la procédure de backstepping classique. Intuitivement, plus les bornes des fonctions $h_1$ et $h_2$ sont grandes plus le bassin d'attraction de la commande locale est petit. Ainsi, il existe un compromis entre la taille des bornes des fonction $h_i$, $i=1,2$ et la validité de l'hypothèse \ref{hyp:inclusion}.

Une condition nécessaire pour que la commande hybride précédemment définie rende l'origine globalement asymptotiquement stable pour le système bouclé est donnée par
$$\A\subseteq \mathbf{B}_{f_h}(\varphi_\ell).$$
Supposons que les hypothèses \ref{hyp:stabilite locale} et \ref{hyp:bornes} sont valides. Si $\A\not\subseteq \mathbf{B}_{f_h}(\varphi_\ell)$, alors l'hypothèse \ref{hyp:inclusion} n'est pas satisfaite. Ainsi, il peut exister des conditions initiales telles que la solution du système $\dx=f_h(x,\varphi_g(x))$ converge vers l'ensemble $\A\setminus \mathbf{B}_{f_h}(\varphi_\ell)$.
Ainsi les trajectoires du système $\dx=f_h(x,\varphi_\ell(x))$ ne convergera pas vers l'origine pour toutes les conditions initiales dans l'ensemble $\A\setminus \mathbf{B}_{f_h}(\varphi_\ell)$.
\end{remark}

Bien que le théorème \ref{thm:commande hybride} donne une commande globalement stabilisante, lorsque les hypothèses \ref{hyp:stabilite locale}-\ref{hyp:inclusion} sont valides,  il n'est pas toujours évident d'effectuer la synthèse d'une commande locale qui satisfasse la contrainte sur le bassin d'attraction. Par ailleurs, la taille de l'ensemble $\A$ dépend du gain de la commande stablisante pour le sous-système en $x_1$. Dans ce travail, nous présentons des conditions sur les gains telles que l'hypothèse \ref{hyp:inclusion} soit satisfaite, quand la commande locale a été déjà donnée par l'hypothèse \ref{hyp:stabilite locale}.

\subsection{Définition du problème}

Notre approche se décompose en deux étapes:
\begin{itemize}
\item Dans un premier temps, nous devons chercher un couple stabilisant, c'est-à-dire, une commande locale $\varphi_\ell$ et une fonction de Lyapunov locale $V_\ell$ avec une estimation du bassin d'attraction pour le système $\dx=f_h(x,\varphi_\ell(x))$. 

\item Dans un second temps, nous devons établir des conditions sur le gain de la commande $\varphi_1$ pour le sous-système défini par \eqref{eq:sous systeme general}, telles que la taille de l'ensemble défini par \eqref{eq:general set:A} soit la plus petite possible afin que l'hypothèse \ref{hyp:inclusion} soit plus facile à satisfaire.
\end{itemize}

Dans cet article, nous n'allons pas analyser un problème général mais nous allons poursuivre l'étude d'une classe particulière de systèmes non-linéaires introduite dans \citep{5991251}.

\section{Une classe de systèmes non-linéaires}\label{sec:classe de problemes}

On rappelle la classe de systèmes non-linéaires présentés dans \citep[Équation (12)]{5991251}:
\begin{equation}\label{eq:systeme de l'exemple}
\left\{\begin{array}{rcl}
\dx_2&=&u\\
\dx_1&=&x_1+x_2+\theta x_1^2+\theta(1+x_1)\sin(u)\\
\end{array}\right.
\end{equation}
avec $\theta>0$ qui est un paramètre libre mais constant pour chaque système.

L'intérêt dans cette classe de systèmes nous permet d'illustrer l'influence de la norme des fonctions $h_i$, $i=1,2$, dans les calculs numériques et dans la faisabilité des hypothèses \ref{hyp:stabilite locale}-\ref{hyp:inclusion}.

Dans  \citep[Lemma 4.1]{5991251} nous avons vérifié que les hypothèses \ref{hyp:stabilite locale}-\ref{hyp:inclusion} sont valides pour tous les systèmes avec $\theta\leq0.001$.
Dans notre étude, nous montrons qu'avec un choix approprié de la commande $\varphi_1$ les mêmes approches sont valides pour une classe plus grande de systèmes (plus précisément pour des valeurs plus grands de $\theta$).

L'idée de notre approche est de considérer un gain $K_1$ qui garantit que l'orientation du plus grand axe de l'ellipsoïde associé aux lignes de niveau de la fonction de Lyapunov soit celle de l'ensemble $\A$.
%

En suivant cette démarche, nous obtenons alors le résultat suivant.

\begin{theorem}\label{thm:result} Pour la classe de systèmes définie par $\theta\leq0.06$, il existe une commande hybride telle que l'origine est globalement asymptotiquement stable pour chaque système en boucle fermée.
\end{theorem}
\begin{proof}

Pour ne pas alourdir la lecture, les calculs intermédiaires seront omis, et nous ne présenterons que les éléments principaux de la démonstration. L'hypothèse \ref{hyp:stabilite locale} est vérifiée, avec la commande
\begin{equation}\label{eq:phil}
\varphi_\ell(x_1,x_2)=-x_1k_1+x_2k_2
\end{equation}
avec $k_1=7+\theta $ et $k_2=-4+4 \theta +\theta  (1+\theta )$. La fonction de Lyapunov pour le système \eqref{eq:systeme de l'exemple} en boucle fermée  avec $\varphi_\ell$ est définie par
\begin{equation}\label{eq:Vl}
\begin{split}
V_\ell(x_1,x_2)=w^TP_\ell w=w^T\begin{bmatrix}5/2&1\\ \\
1&1/2
\end{bmatrix}w
\end{split}
\end{equation}
$w=[x_1-\theta x_2, x_2]^T$. La constante $c_\ell$ vaut
\begin{equation}\label{eq:borne vl}
c_\ell=\left(\dfrac{2-\theta p_1(\theta)}{\theta p_2(\theta)}\right)^2,\quad\forall\theta<\theta_1.\footnote{Les polynômes $p_1$ et $p_2$ sont donnés par
\begin{equation*}
\begin{array}{rcl}
	p_1(\theta)&=&-396+2308 \theta +9768 \theta ^2+1440 \theta ^3\\
	p_2(\theta)&=&792+7000 \theta +20856 \theta ^2+21672 \theta ^3+2160 \theta ^4.
\end{array}
\end{equation*} et $\theta_1=0,118462$ est la plus petite racine positive de l'équation $\xi(\theta)=-2+\theta p_1(\theta)$.}
\end{equation}

L'hypothèse \ref{hyp:bornes} est vérifiée avec $\varphi_1(x_1)=-2.7456x_1-\theta x_1^2$, $V_1(x_1)=x_1^2/2$, $\alpha(s)=3.4912s$, $\Psi(x_1,x_2)=\theta(1+|x_1|)$, $\varepsilon=0.89$ et $M=0.02$. Par ailleurs, la proposition \ref{prop:commande globale} assure l'existence de la commande donnée par \eqref{eq:commande globale}, avec $c=a=10$,
\begin{equation}\label{eq:commande globale:exemple}
\begin{split}
\varphi_g(x_1,x_2)= -(2.7456+2\theta x_1)(x_1+x_2+\theta x_1^2)\\
-\frac{x_1}{2K_V}-\frac{x_2-\varphi_1(x_1)}{K_V}(c+\frac{c}{4}\Delta^2(x_1,x_2)),
\end{split}
\end{equation}
et
\begin{equation*}
\begin{split}
\Delta(x_1,x_2)=\theta K_V(1+|x_1|)(|x_1|/K_V+1\\
+|2.7456+&2\theta x_1|)
\end{split}
\end{equation*}
telle que l'ensemble $\A=\{(x_1,x_2):|x_1|\leq0.2,x_2=\varphi_1(x_1)\}$ est globalement pratiquement stable pour le système \eqref{eq:systeme de l'exemple} en boucle fermée avec $\varphi_g$.

Pour l'hypothèse \ref{hyp:inclusion}, on trouve qu'elle est vérifiée, $\forall\theta< 0.0607418$. En particulier, avec $\theta=0.06$ on a
\begin{equation*}
\begin{split}
\max_{(x_1,x_2)\in\A}V_\ell(x_1,x_2)=V_\ell(0.2,\varphi_1(0.2))\\
=0.030807<0.0390824=c_\ell
\end{split}
\end{equation*}

Avec le théorème \ref{thm:commande hybride}, nous obtenons donc une loi de commande hybride nous permettant de conclure la preuve du théorème \ref{thm:result}.
\end{proof}

\section{Simulations}\label{sec:simulations}

Dans cette section, nous présentons des simulations du système précédant.
Nous considérons comme condition initiale des points au bord d'une boule de rayon 2 et nous exécutons une simulation du système \eqref{eq:systeme de l'exemple} avec $\theta=0.06$ et bouclé avec $\mathbb{K}$.

Comme attendu, le système converge vers l'origine.
Pour avoir une illustration claire nous montrons l'évolution temporelle de $V_\ell$  (figure \ref{sim:theta 006:Vl}) et des composantes de la solution $(x,q)$ du système bouclé avec $\mathbb{K}$ seulement pour la condition initiale $(x_1,x_2,q)=(2,0,1)$ (figure \ref{sim:theta 006:x1x2q}). Pour les autres points qui appartiennent à la boule le comportement est similaire, comme le montre la figure \ref{sim:theta 006:traj} qui est la trajectoire de $5$ points dans l'espace d'états.

\begin{figure}[htbp!]
\begin{center}
\includegraphics[width=\textwidth]{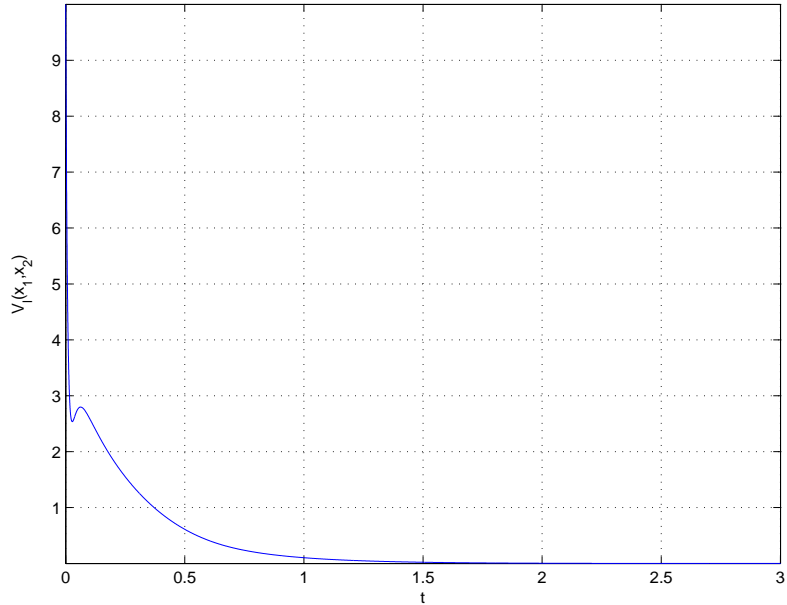}
\caption{Evolution de \eqref{eq:Vl} pour le système \eqref{eq:systeme de l'exemple} avec $\theta=0.06$, en boucle fermée avec $\mathbb{K}$.}
\label{sim:theta 006:Vl}
\end{center}
\end{figure}

\begin{figure}[htbp!]
\begin{center}
\includegraphics[width=\textwidth]{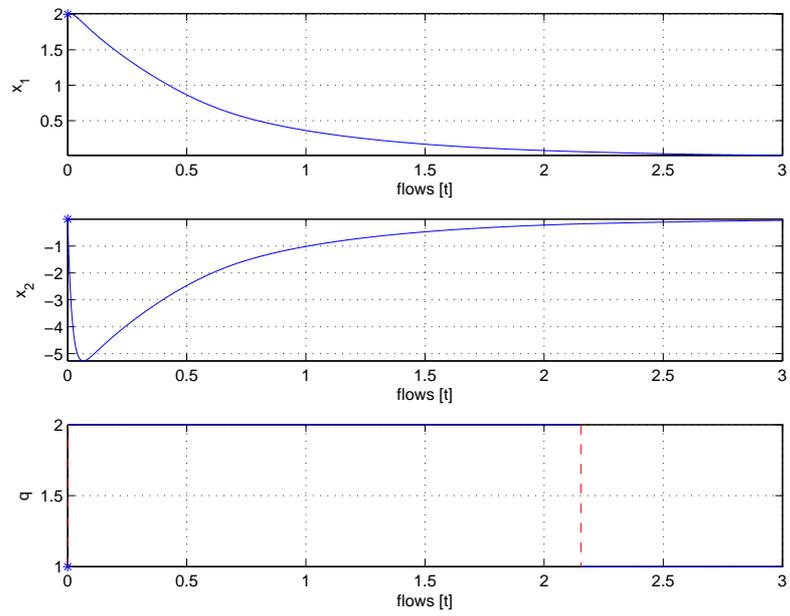}
\caption{Evolution des composantes de la solution du système \eqref{eq:systeme de l'exemple} avec $\theta=0.06$, en boucle fermée avec $\mathbb{K}$.}
\label{sim:theta 006:x1x2q}
\end{center}
\end{figure}

\begin{figure}[htbp!]
\begin{center}
\includegraphics[width=\textwidth]{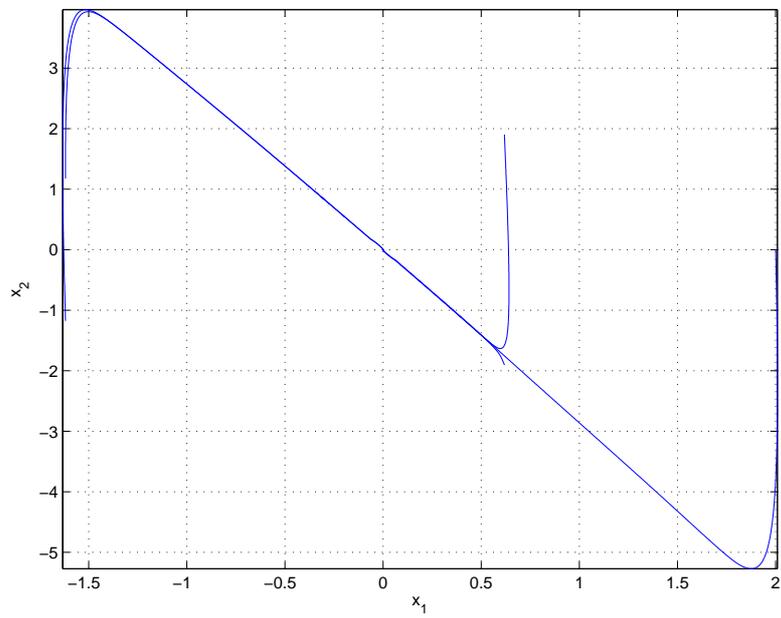}
\caption{Trajectoires des solutions pour différentes conditions initiales choisies au bord d'une boule de rayon 2.}
\label{sim:theta 006:traj}
\end{center}
\end{figure}

\section{Conclusion et perspectives}

Dans ce document, nous montrons que par un choix approprié de la loi de commande locale $\varphi_1$ pour le sous-système $x_1$, cela permet de vérifier l'hypothèse \ref{hyp:inclusion} pour une classe plus grande de systèmes que ce que nous avions obtenu lors de nos travaux précédents dans \citep{5991251}.
A présent, l'objectif est donc de trouver une démarche constructive de sélection de ce contrôleur local
afin de garantir que le bassin d'attraction de celui-ci contiendra l'ensemble $\A$ et par conséquent de vérifier l'hypothèse \ref{hyp:inclusion}.
\bibliographystyle{IEEEtranN}
\bibliography{bibliographie}
\end{document}